\documentclass[11pt,UTF8]{article}
\usepackage[top=28mm, bottom=28mm, left=35mm, right=35mm]{geometry}

\usepackage{proof}
\usepackage[UKenglish]{babel}
\usepackage{amsmath}
\usepackage{amssymb}
\usepackage{latexsym}
\usepackage{amsthm}
\usepackage{amsfonts}
\usepackage{color}
\usepackage{hyperref}
\usepackage{url}

\hypersetup{
   colorlinks,
   citecolor=blue,
   filecolor=black,
   linkcolor=red,
   urlcolor=black
 }%

\usepackage[all]{xy}
\usepackage{tikz} 
\usepgflibrary{arrows}

\newtheorem{theorem}{Theorem}
\newtheorem{lemma}[theorem]{Lemma}
\newtheorem{claim}[theorem]{Claim}
\theoremstyle{definition}
\newtheorem{corollary}[theorem]{Corollary}
\newtheorem{example}[theorem]{Example}
\newtheorem{fact}[theorem]{Fact}
\newtheorem{definition}[theorem]{Definition}

\def\Ind{\setbox0=\hbox{$x$}\kern\wd0\hbox to 0pt{\hss$\mid$\hss}
\lower.9\ht0\hbox to 0pt{\hss$\smile$\hss}\kern\wd0}
\def\Notind{\setbox0=\hbox{$x$}\kern\wd0\hbox to 0pt{\mathchardef
\nn=12854\hss$\nn$\kern1.4\wd0\hss}\hbox to
0pt{\hss$\mid$\hss}\lower.9\ht0 \hbox to
0pt{\hss$\smile$\hss}\kern\wd0}
\def\ind{\mathop{\mathpalette\Ind{}}}
\def\nind{\mathop{\mathpalette\Notind{}}}

\title{Pseudofinite  $H$-structures and groups definable in supersimple $H$-structures}
\author{Tingxiang Zou\footnote{This author is supported by the China Scholarship Council and partially supported by ValCoMo (ANR-13-BS01-0006).} (\href{mailto:zou@math.univ-lyon1.fr}{\url{zou@math.univ-lyon1.fr}})}
\date{}

\begin{document}
\maketitle
\begin{abstract}
In this paper we explore some properties of $H$-structures which are introduced in \cite{Berenstein2016}.

We describe a construction of $H$-structures based on one-dimensional asymptotic classes which preserves pseudo-finiteness. That is, the $H$-structures we construct are ultraproducts of finite structures.

We also prove that under the assumption that the base theory is supersimple of $SU$-rank one, there are no new definable groups in $H$-structures. This improves the corresponding result in \cite{Berenstein2016}.
\end{abstract}

\section{Introduction}
$H$-structures are introduced in \cite{Berenstein2016}. They are based on a geometric theory, where algebraic closure satisfies the exchange property and $\exists^\infty$ is eliminated. When a dense and co-dense independent subset is added to a model of this theory, the resulting structure is an $H$-structure. Strongly minimal theories, supersimple $SU$-rank one theories and superrosy thorn-rank one theories with elimination of $\exists^{\infty}$ are examples of geometric theories. In these cases, the corresponding $H$-structures preserve $\omega$-stability, supersimplicity or superrosiness and the rank is either one or $\omega$.

In the following, we will recall the definition of $H$-structures and some of their main properties. 

Let $T$ be a complete geometric theory in a language $\mathcal{L}$. Let $H$ be a unary predicate and put $\mathcal{L}_H=\mathcal{L}\cup\{H\}$. Let $M\models T$; we say that $A\subseteq M$ is finite dimensional if $A\subseteq \mbox{acl}_{\mathcal{L}}(a_1,\ldots,a_n)$ for some $a_1,\ldots,a_n\in M$. For a tuple $a$ and a set of parameters $A$, we write $\mbox{dim}_{\mbox{acl}_{\mathcal{L}}}(a/A)$  as the length of a maximal $\mbox{acl}_{\mathcal{L}}$-independent subtuple of $a$ over $A$.

\begin{definition}
We say that $(M,H(M))$ is an \textsl{H-expansion of $M$} \footnote{It is just called an $H$-structure in \cite{Berenstein2016}, we add this terminology to be more precise about the base theory or the base model.} if:
\begin{enumerate}
\item
$M\models T$;
\item
$H(M)$ is an $\mbox{acl}_{\mathcal{L}}$-independent subset of $M$;
\item
(Density/coheir property) If $A\subseteq M$ is finite dimensional and $q\in S_1(A)$ is non-algebraic, there is $a\in H(M)$ such that $a\models q$;
\item
(Extension property) If $A\subseteq M$ is finite dimensional and $q\in S_1(A)$ is non-algebraic, then there is $a\in M$, $a\models q$ and $a\not\in \mbox{acl}_{\mathcal{L}}(A\cup H(M))$.
\end{enumerate}

Equivalently, we can replace density and extension properties with the following more general ones:
\begin{itemize}
\item
(Generalised density/coheir property) If $A\subseteq M$ is finite dimensional and $q\in S_n(A)$ has dimension $n$, then there is $a\in H(M)^n$ such that $a\models q$;
\item
(Generalised extension property) If $A\subseteq M$ is finite dimensional and $q\in S_n(A)$ is non-algebraic, then there is $a\in M^n$, $a\models q$ and $$\mbox{dim}_{\mbox{acl}_{\mathcal{L}}}(a/A, H(M))=\mbox{dim}_{\mbox{acl}_{\mathcal{L}}}(a/A).$$
\end{itemize}

A structure $\mathcal{M}$ is called an \textsl{$H$-structure} if it is an $H$-expansion of some model of a geometric theory.
\end{definition}

$H$-structures are closely related to lovely pairs, where, instead of an independent subset, a dense and co-dense elementary substructure is added. We recall the definition of lovely pairs in the special case that the base theory is geometric, see \cite{berenstein2010lovely}.

\begin{definition}
Let $T$ be a geometric theory in a language $\mathcal{L}$ and let $\mathcal{L}_P$ be the expansion of $\mathcal{L}$ by a unary predicate $P$. An $\mathcal{L}_P$-structure $(M,N)$ is a \textsl{lovely pair of models of $T$}, if 
\begin{enumerate}
\item
$M\models T$;
\item
$N$ is an $\mathcal{L}$-elementary submodel of $M$;
\item
(Density/coheir property) If $A\subseteq M$ is finite dimensional and $q\in S_1(A)$ is non-algebraic, there is $a\in N$ such that $a\models q$;
\item
(Extension property) If $A\subseteq M$ is finite dimensional and $q\in S_1(A)$ is non-algebraic, then there is $a\in M$, $a\models q$ and $a\not\in \mbox{acl}_{\mathcal{L}}(A\cup N)$.
\end{enumerate}
\end{definition}
\begin{fact}\cite{Berenstein2016}, \cite{berenstein2010lovely}.
Properties of $H$-structures and lovely pairs.

Let $T$ be a complete geometric theory in a language $\mathcal{L}$.
\begin{itemize}

\item
$H$-expansions of models of $T$ exist and all of them are $\mathcal{L}_H$-elementary equivalent. Let $T_{H}$ be the corresponding theory. Similarly, lovely pairs of models of $T$ exist, and all of them are $\mathcal{L}_P$-elementary equivalent. 
\item
If the geometry of $T$ is nontrivial and $T$ is strongly minimal/supersimple/superrosy of rank 1, then $T_{H}$ is $\omega$-stable/supersimple/superrosy of rank $\omega$.
\item
Let $(M,H(M))$ be an $H$-structure. Then $(M,\mbox{acl}_{\mathcal{L}}(H(M)))$ is a lovely pair.
\end{itemize}
\end{fact}

Consider the theory of pseudofinite fields. It is supersimple of $SU$-rank one. By the fact above, $H$-expansions and lovely pairs of pseudofinite fields exist. However, the proof of existence uses general model theoretic techniques such as saturated models and union of chains. It is not clear whether it is possible to have $H$-expansions or lovely pairs of pseudofinite fields that are ultraproducts of finite structures.

The answer turns out to be negative for lovely pairs.

\begin{lemma}\label{lovely-pair}
If $(K,k)$ is a lovely pair of pseudofinite fields, then it is not pseudofinite.\footnote{This was already noticed by Gareth Boxall (private communication).}
\end{lemma}
\begin{proof}
Let $(K',k')=\prod_{i\in I}(K'_i,k'_i)/\mathcal{U}$ be a pair of pseudofinite fields with $\mbox{char}(K')=\mbox{char}(k')$ such that $k'_i\subsetneq K'_i$ are finite fields for any $i\in I$.

Suppose $\mbox{char}(K')\neq 2$. We will show that there are $a_1,a_2\in K'$ and $\varphi(x;y_1,y_2)$ in the language of rings such that $\varphi(x;a_1,a_2)$ is non-algebraic, but there is no $b\in k'$ such that $\varphi(b;a_1,a_2)$ holds. However, as $(K,k)$ is a lovely pair, the following holds in $(K,k)$: $$\forall y_1\forall y_2(\exists^{\infty}x~\varphi(x;y_1,y_2) \to \exists z\in k~\varphi(z;y_1,y_2)).$$ Therefore, $(K,k)$ is not elementary equivalent to $(K',k')$.

As $\mbox{char}(K')\neq 2$, we may assume that $\mbox{char}(K_i)\neq 2$ for all $i\in I$. For any $i\in I$ take $\sigma_i\in Gal(K'_i/k_i')$ with $\sigma_i\neq id$. Let $a_{i_1},a_{i_2}\in K'_i$ be such that $\sigma_i(a_{i_1})=a_{i_2}$ and $a_{i_1}\neq a_{i_2}$. Let $\sigma=(\sigma_i)_{i\in I}/\mathcal{U}$, $a_1:=(a_{i_1})_{i\in I}/\mathcal{U}$ and $a_2:=(a_{i_2})_{i\in I}/\mathcal{U}$. Then $a_1\neq a_2$, $\sigma(a_1)=a_2$ and $k'\subseteq \mbox{Fix}(\sigma)$. Define $$\varphi(x;y_1,y_2):=(\exists z~z^2=x-y_1)~\land~\neg(\exists z~z^2=x-y_2).$$ We claim that $\varphi(x;a_1,a_2)$ is non-algebraic in $K'$. Since $\mbox{char}(K_i')\neq 2$ for any $i\in I$, we have $\{x^2:x\in K_i'\}\subsetneq K_i'$. Let $e_i$ be such that there is no $x\in K_i'$ with $x^2=e_i$. Then by \cite[Proposition 4.3]{duret}, the ideal generated by $\{(X_1)^2-(X-a_{i_1});~(X_2)^2-e_i(X-a_{i_2})\}$ is absolutely prime and does not contain $X-a_{i_1}$ or $X-a_{i_2}$. Let $V$ be the corresponding irreducible variety. Then $V$ has dimension 1; by the Lang-Weil estimate $|V\cap K_i'|\approx |K_i'|$. We claim that $K_i\models \varphi(x;a_{i_1},a_{i_2})$ for any $(x_1,x_2,x)\in V\cap K_i'$ with $x\neq a_{i_2}$. Since if not, there is some $x_3$ such that $x-a_{i_2}=(x_3)^2$. As $x\neq a_{i_2}$, we have $x_3\neq 0$. Then $e_i=(\frac{x_2}{x_3})^2$, contracting that $e_i$ is not a square-root. Therefore, we can define a function $$\tau_i:(V\cap K_i')\setminus \{(x_1,x_2,a_{i_2}):x_1,x_2\in K_i'\}\to\varphi(K_i';a_{i_1},a_{i_2})$$ by $\tau_i(x_1,x_2,x):=x$. As $\mbox{char}(K_i')\neq 2$, it is easy to see that $\tau_i$ is a four-to-one function. By that $|V\cap K_i'|\approx |K_i'|$, we conclude that $$|\varphi(K_i';a_{i_1},a_{i_2})|\approx \frac{1}{4}|V\cap K_i'|.$$ Thus, $\varphi(x;a_1,a_2)$ is non-algebraic.

On the other hand, for any $b\in k'$ we have $$\exists z(z^2=b-a_1)\Longleftrightarrow\exists z(\sigma(z^2)=\sigma(b-a_1))\Longleftrightarrow\exists z( \sigma(z)^2=b-a_2)\Longleftrightarrow\exists z(z^2=b-a_2).$$ Therefore, there is no $b\in k'$ such that $\varphi(b;a_1,a_2)$ holds.

The case of $\mbox{char}(K')=2$ is similar, using cubes instead of squares (and possibly going to some finite extension of $K'$).
\end{proof}

In view of the close connection between $H$-structures and lovely pairs, we might expect $H$-expansions of pseudofinite fields never to be pseudofinite. Luckily, this is not so. In fact, we can see from the proof above that the reason $(K',k')$ is not a lovely pair is the existence of a nontrivial automorphism $\sigma$ of $K'$ that fixes $k'$. In the case of $H$-expansions, instead of a subfield we only need to add a subset. Intuitively, we might be able to choose a pseudofinite set large enough such that no non-trivial automorphism can fix all the points in this set.

\begin{definition}
Let $T$ be a geometric theory in a language $\mathcal{L}$.
Let $\mathcal{M}=\prod_{i\in I}M_i/\mathcal{U}\models T$ be an infinite ultraproduct of finite structures. We call an $H$-expansion $(\mathcal{M}, H(\mathcal{M}))$ an \textsl{exact pseudofinite $H$-expansion of $\mathcal{M}$} if $(\mathcal{M}, H(\mathcal{M}))=\prod_{i\in I}(M_i,H_i)/\mathcal{U}$ with $H_i\subseteq M_i$ for all $i\in I$.
\end{definition}
\noindent\emph{Remark:} Let $\mathcal{M}=\prod_{i\in I}M_i/\mathcal{U}\models T$ be pseudofinite. Then an arbitrary pseudofinite $H$-expansion need not to be exact, since it need not be this particular ultraproduct. For example, let $\mathcal{U}$ be a nonprincipal ultrafilter on $\mathbb{N}$ and $V=\prod_{i\in \mathbb{N}}V_n/\mathcal{U}$ an ultraproduct of finite vector spaces over $\mathbb{F}_2$ such that $\lim_{n\in\mathbb{N}}\mbox{dim}(V_n)=\infty$. It is easy to build an exact pseudofinite $H$-expansion of $V$ by choosing an independent set $H_n\subseteq V_n$ for each $n\in\mathbb{N}$ with $\displaystyle{\lim_{n\in\mathbb{N}}\mbox{dim}(H_n)=\lim_{n\in\mathbb{N}}\mbox{codim}(H_n)=\infty}$ and put $(V,H)=\prod_{n\in\mathbb{N}}(V_n,H_n)/\mathcal{U}$. Let $H'\subseteq V$ be a countable independent set of $V$. Then $(V,H')$ is pseudofinite $H$-expansion of $V$ as $(V,H')\equiv (V,H)$. But $(V,H')$ is not $\aleph_1$-saturated, hence cannot be an ultraproduct over non-principal ultrafilters. Thus $(V,H')$ is not exact.


Let $\mathcal{C}$ be a one-dimensional asymptotic class and $\mathcal{M}$ be an infinite ultraproduct of members of $\mathcal{C}$. In section \ref{H-fields} we show that exact pseudofinite $H$-expansions of $\mathcal{M}$ always exist. In particular, pseudofinite $H$-expansions of pseudofinite fields do exist.

Section \ref{groups} deals with definable groups in $H$-structures. Our motivation is to classify definable groups in $H$-expansions of pseudofinite fields. There are some results about definable groups in $H$-structures when the base theory is superstable in \cite{Berenstein2016} using the group configuration theorem. The problem to generalise these results is that in simple (even in supersimple) theories, there is no nice version of the group configuration theorem available in general. However, pseudofinite fields are exceptional: the group configuration theorem for pseudofinite fields has essentially been given in \cite{hrushovski1994groups}. We can easily deduce that definable groups in $H$-expansions of pseudofinite fields are virtually isogenous to algebraic groups.

However, this is not very satisfactory. It is of course the best one could get when one compares definable groups in $H$-expansions of pseudofinite fields with algebraic groups. But as has been noticed in \cite{Berenstein2016}, ``since the geometry on $H$ is trivial, we expected adding $H$ should not introduce new definable groups". With the help of the group chunk theorem in simple theories (see Fact \ref{prop-gpchunk}) we give a more satisfactory answer, namely, there are no new definable groups in $H$-structures when the base theory is supersimple of $SU$-rank one. Notably, Eleftheriou also got a same classification of definable groups in $H$-structures in the setting of o-minimal theories using the similar strategy, see \cite[Theorem 1.2]{Eleftheriou2018character}.

\section{Pseudofinite $H$-structures}\label{H-fields}
This section deals with pseudofinite $H$-structures built from one-dimensional asymptotic classes. 

One-dimensional asymptotic classes are classes of finite structures with a nicely behaved dimension and counting measure on all families of uniformly definable sets. They are introduced in \cite{macpherson2008one} inspired by the class of finite fields.

We recall the definition of a one-dimensional asymptotic class and list some examples. 

\begin{definition}\label{lem-counting}
Fix a language $\mathcal{L}$. A class $\mathcal{C}$ of finite $\mathcal{L}$-structures is called a \textsl{one-dimensional asymptotic class} if the following holds for every $m\in\mathbb{N}$ and every formula $\varphi(x;\bar{y})$ with $|\bar{y}|=m$:
\begin{enumerate}
\item
There is a positive constant $C$ and a finite set $E\subseteq \mathbb{R}^{>0}$ such that for any $M\in\mathcal{C}$ and $\bar{b}\in M^m$, either $|\varphi(M;\bar{b})|<C$ or there is $\mu\in E$ with $$||\varphi(M;\bar{b})|-\mu |M||<C\cdot |M|^{\frac{1}{2}}.$$
\item
For every $\mu\in E$ there is an $\mathcal{L}$-formula $\varphi_{\mu}(\bar{y})$ such that for any $M\in\mathcal{C}$ and $\bar{b}\in M^m$
\begin{center}
$M\models \varphi_{\mu}(\bar{b})$ if and only if $||\varphi(M;\bar{b})|-\mu |M||<C\cdot |M|^{\frac{1}{2}}$.
\end{center}
\end{enumerate}
\end{definition}

Examples of one-dimensional asymptotic classes are:
\begin{itemize}
\item
The class of finite fields.
\item
The class of finite-dimensional vector spaces over finite fields.
\item
The class of finite cyclic groups.
\end{itemize}

\begin{fact}(\cite[Lemma 4.1]{macpherson2008one})
Let $\mathcal{C}$ be a one-dimensional asymptotic class and $M$ an infinite ultraproduct of members of $\mathcal{C}$. Then $Th(M)$ is supersimple of $SU$-rank 1.
\end{fact}
In particular, the theory of any infinite ultraproduct of members of a one-dimensional asymptotic class is a model of a geometric theory, and we will show that it always allows an exact pseudofinite $H$-expansion.

\begin{definition}
Let $\mathcal{C}$ be a one-dimensional asymptotic class in a language $\mathcal{L}$. Let $\varphi(x;\bar{y})$ ($\bar{y}$ non-empty) be an $\mathcal{L}$-formula and $E\subseteq \mathbb{R}^{>0}$ be as in Definition \ref{lem-counting}. Put $$\psi_{\varphi}(\bar{y}):=\bigvee_{\mu\in E}\varphi_{\mu}(\bar{y}).$$ For a structure $M\in\mathcal{C}$ and a subset $X\subseteq M$, we say \textsl{$X$ covers $\psi_{\varphi}(\bar{y})$ in $M$} if the following holds: $$\bigcup_{x\in X}\varphi(x;M^{|\bar{y}|})\supseteq \psi_{\varphi}(M^{|\bar{y}|}).$$

Let $\phi(x;\bar{y})$ be a formula. Suppose $\phi(x;\bar{y})$ is algebraic ($\bar{y}$ can be empty) over any $\bar{y}$. For a structure $M\in\mathcal{C}$ and a linearly-ordered subset $X\subseteq M$, we say that \textsl{$X$ avoids $\phi(x;\bar{y})$ in $M$} if there is no $x,x_1,\ldots, x_{|\bar{y}|}\in X^{|\bar{y}|+1}$ such that $x>\max\{x_1,\ldots,x_{|\bar{y}|}\}$ and $$M\models \phi(x;x_1,\ldots,x_{|\bar{y}|}).$$
\end{definition}

Let $\mathcal{M}$ be an infinite ultraproduct of members of $\mathcal{C}$. For any $\varphi(x,\bar{y})$ and $\bar{a}\in\mathcal{M}^{|\bar{y}|}$, if $\mathcal{M}\models \psi_{\varphi}(\bar{a})$, then there is $\mu\in E$ such that $|\varphi(\mathcal{M},\bar{a})|\approx\mu|\mathcal{M}|.$ As $\mu>0$ and $\mathcal{M}$ is infinite, we get $\varphi(\mathcal{M},\bar{a})$ is infinite. On the other hand, if $\mathcal{M}\models \neg\psi_{\varphi}(\bar{a})$, then by the definition one-dimensional asymptotic class, there must be some $C\in\mathbb{N}$ such that $|\varphi(\mathcal{M},\bar{a})|\leq C$. Therefore, $\psi_{\varphi}(\bar{y})$ defines the set of $\bar{a}$ such that $\varphi(x,\bar{a})$ is non-algebraic in any infinite ultraproduct of members of $\mathcal{C}$.

\begin{lemma}\label{lem-cover}
Let $\mathcal{C}$ be a one-dimensional asymptotic class, $\Gamma$ be a finite set of algebraic formulas of the form $\phi(x;\bar{z})$ ($\bar{z}$ could be empty) and $\Delta$ any finite set of formulas of the form $\varphi(x;\bar{y})$ (the length of $\bar{y}$ can vary and $\bar{y}$ is non-empty). Then there are $N_{\Delta,\Gamma}\in\mathbb{N}$ and $C_{\Delta,\Gamma}\in \mathbb{R}^{>0}$ such that the following holds:

For any $M\in\mathcal{C}$ with $|M|\geq N_{\Delta,\Gamma}$, there exists $(H_{\Delta,\Gamma}(M),\leq )$ with $H_{\Delta,\Gamma}(M)\subseteq M$ and $\left|H_{\Delta,\Gamma}(M)\right|\leq C_{\Delta,\Gamma}\cdot\log |M|$ such that for any $\varphi(x;\bar{y})\in \Delta$ and $\phi(x;\bar{z})\in\Gamma$, we have $H_{\Delta,\Gamma}(M)$ covers $\psi_{\varphi}(\bar{y})$ and avoids $\phi(x;\bar{z})$ in $M$. 

In particular, $|M|\geq N_{\Delta,\Gamma}$ should imply the equation (\ref{hm}) and the inequality (\ref{star}), which are defined throughout the proof.
\end{lemma}
\begin{proof}
By Definition \ref{lem-counting}, for each $\varphi(x;\bar{y})\in \Delta$ there are finitely many $\mu_{0,\varphi},\ldots,\mu_{k_{\varphi},\varphi}>0$ and $C_{\varphi}\in\mathbb{R}$, such that for any $M\in\mathcal{C}$ and $\bar{a}\in M^{|\bar{y}|}$,  $$\psi_{\varphi}(\bar{a})~~\Longrightarrow~~\bigvee_{j\leq k_{\varphi}}(||\varphi(M;\bar{a})|-\mu_{j,\varphi} \cdot |M| |<C_{\varphi}\cdot|M|^{\frac{1}{2}}).$$ Take $0<\mu<\min\{\mu_{0,\varphi},\ldots,\mu_{k_{\varphi},\varphi}:\varphi\in \Delta\}$. Let $$\mathcal{C}_\mu:=\bigcap_{\varphi\in \Delta}\{M\in\mathcal{C}:\text{ for any }\bar{a}, ~\psi_{\varphi}(\bar{a})\text{ implies }|\varphi(M;\bar{a})|\geq \mu\cdot |M|\}.$$ We claim that there is some $N\in\mathbb{N}$ such that for any $M\in\mathcal{C}$ and $|M|>N$, we have $M\in \mathcal{C}_\mu$. Otherwise, there are $\varphi(x;\bar{y})\in \Delta,\mu_{i_0,\varphi}>0$ and $\{M_i\in\mathcal{C},\bar{a}_i\in M_i^{|\bar{y}|}:i\in\mathbb{N}\}$ such that the following holds:
\begin{itemize}
\item
$\lim_{i\to\infty}|M_i|=\infty$;
\item
$M_i\models \varphi_{\mu_{i_0,\varphi}}(\bar{a}_i)$ for each $i\in\mathbb{N}$;
\item
$|\varphi(M_i;\bar{a}_i)|<\mu \cdot|M_i|<\mu_{i_0,\varphi}\cdot |M_i|$ for each $i\in\mathbb{N}$.
\end{itemize}
Therefore, $$\mu_{i_0,\varphi}\cdot|M_i|-|\varphi(M_i;\bar{a}_i)|>(\mu_{i_0,\varphi}-\mu)\cdot|M_i|=(\mu_{i_0,\varphi}-\mu)\cdot|M_i|^{\frac{1}{2}}\cdot|M_i|^{\frac{1}{2}}.$$
By the definition of one-dimensional asymptotic class, there is some $C_{\varphi}>0$ such that $$||\varphi(M_i;\bar{a}_i)|-\mu_{i_0,\varphi}\cdot|M_i||< C_{\varphi}\cdot |M_i|^{\frac{1}{2}}.$$ Since $\lim_{i\to\infty}(\mu_{i_0,\varphi}-\mu)\cdot|M_i|^{\frac{1}{2}}=\infty$, there is clearly a contradiction.

Assume $\Delta=\{\varphi_1(x;\bar{y}_1),\ldots,\varphi_n(x;\bar{y}_n)\}$. Fix any $M\in\mathcal{C}$ with $|M|> N$, for $1\leq i\leq n$, define inductively the following sets: $X^i_j,L^i_j,H^i_j\subseteq M$ and $Y^i_j\subseteq \psi_{\varphi_i}(M^{|\bar{y}_i|})$.
\begin{itemize}
\item
$Y^1_0:=\psi_{\varphi_1}(M^{|\bar{y}_1|})$;
\item
$X^1_0:=H^1_0:=L^1_0:=\emptyset$;
\end{itemize}

Suppose $Y^i_j,X^i_j,H^i_j,L^i_j$ are defined. There are two cases.
\begin{itemize}
\item
If $Y^i_j=\emptyset$ and $i<n$, define
\begin{itemize}
\item
$Y^{i+1}_0:=\psi_{\varphi_{i+1}}(M^{|\bar{y}_{i+1}|})$;
\item
$X^{i+1}_0:=L^{i+1}_0:=\emptyset$;
\item
$H^{i+1}_0:=H^{i}_j$.
\end{itemize}
\item
If $Y^i_j\neq\emptyset$, define
\begin{itemize}
\item
$L^i_{j+1}:=\bigcup_{\phi(x;\bar{z})\in\Gamma}\{a\in M: \exists\bar{z}\in (H^i_j)^{|\bar{z}|},M\models\phi(a;\bar{z})\}\cup\bigcup_{\phi'(x)\in\Gamma}\phi'(M)$.
\item
$X^i_{j+1}:=M\setminus(H^i_j\cup L^i_{j+1})$.
\item
Choose an element $h^i_{j+1}$ in $X^i_{j+1}$ such that $\varphi_i(h^i_{j+1};Y^i_j)$ has the maximal cardinality among $\{\varphi_i(a;Y^i_j):a\in X^i_{j+1}\}$.
\item
$H^i_{j+1}:=H^i_j\cup\{h^i_{j+1}\}$ and $Y^i_{j+1}=Y^i_j\setminus \varphi_i(h^i_{j+1};Y^i_j)$.
\end{itemize}
\end{itemize}

The construction stops either when $Y^n_j$ is empty, that is $H^i_j$ covers $\psi_{\varphi_i}(\bar{y}_i)$ for any $1\leq i\leq n$, or when $Y^i_j\neq\emptyset$ and $X^i_{j+1}=\emptyset$ for some $1\leq i\leq n$ and $j\in\mathbb{N}$.

Let $Y^1_0,\ldots,Y^i_j$ be a maximal sequence of the construction. Define $H_{\Delta,\Gamma}(M):=H^i_j$ if $i=n$ and $Y^i_j=\emptyset$.

\begin{claim}\label{claim-cover}
There is $N_{\Delta,\Gamma}\in\mathbb{N}$ such that if $M\in\mathcal{C}$ and $|M|\geq N_{\Delta,\Gamma}$, then $H_{\Delta,\Gamma}(M)$ is always defined.
\end{claim}

\begin{proof}
Suppose $|M|>N$ and $M\in\mathcal{C}$. We first estimate the size of $Y^i_{j+1}$ in terms of $Y^i_j$ when the latter is not empty during the construction of $\{H^i_j,Y^i_j,L^i_j,X^i_j:i\leq n,j\geq 0\}$.

Suppose all $\phi(x;\bar{z})\in\Gamma$ have no more than $C$-many solutions over any parameter $\bar{z}$ ($\bar{z}$ can be empty). Let $C_\Gamma:=C\cdot|\Gamma|$ and $k_0:=\max\{|\bar{z}|:\phi(x;\bar{z})\in \Gamma\}$. Then 
$|L^i_{j+1}|\leq C_{\Gamma}\cdot (|H^i_j|+1)^{k_0}.$\footnote{Since we need to include the algebraic elements over $\emptyset$ defined by formulas in $\Gamma$, it can be that $H^i_j=\emptyset$ but $L^i_{j+1}\neq\emptyset$, that's the reason we put $|H^i_j|+1$ instead of $|H^i_j|$.}

Therefore, 
\begin{equation}\label{Xij}
|X^i_{j+1}|\geq |M|-C_\Gamma\cdot (|H^i_j|+1)^{k_0}-|H^i_j|.
\end{equation}

By construction, $Y^i_{j+1}=Y^i_j\setminus\{\varphi_i(h^i_{j+1};Y^i_j)\}$. As $\varphi_i(h^i_{j+1};Y^i_j)$ is maximal among $\{\varphi_i(a;Y^i_j):a\in X^i_{j+1}\}$, we get $$|\varphi_i(h^i_{j+1};Y^i_j)|\geq \frac{|\bigcup_{a\in X^i_{j+1}}\{(a,\bar{y}):\bar{y}\in\varphi_i(a;Y^i_j)\}|}{|X^i_{j+1}|}\geq\frac{|\bigcup_{a\in X^i_{j+1}}\{(a,\bar{y}):\bar{y}\in\varphi_i(a;Y^i_j)\}|}{|M|}.$$

Let $\mbox{Tot}:=\bigcup_{x\in (M\setminus H^i_j)}\{(x,\bar{y}):\bar{y}\in\varphi_i(x;Y^i_j)\}$, then $$\bigcup_{a\in X^i_{j+1}}\{(a,\bar{y}):\bar{y}\in\varphi_i(a;Y^i_j)\}=\mbox{Tot}\setminus \bigcup_{a\in L^i_{j+1}}\{(a,\bar{y}):\bar{y}\in\varphi_i(a;Y^i_j)\}.$$

As $M\in\mathcal{C}_{\mu}$, for each $\bar{y}\in Y^i_j$ we have $|\varphi_i(M;\bar{y})|\geq \mu \cdot|M|$. And by the definition of $Y^i_j$, for any $\bar{y}\in Y^i_j$, if $M\models \varphi_i(a;\bar{y})$, then $a\not\in H^i_j$. Hence, $|\mbox{Tot}|\geq \mu \cdot|M|\cdot|Y^i_j|$. On the other hand, $$|\bigcup_{a\in L^i_{j+1}}\{(a,\bar{y}):\bar{y}\in\varphi_i(a;Y^i_j)\}|\leq \left|L^i_{j+1}|\cdot|Y^i_j\right|\leq C_\Gamma\cdot (|H^i_j|+1)^{k_0}\cdot |Y^i_j|.$$ Hence, $$|\varphi_i(h^i_{j+1};Y^i_j)|\geq \frac{\mu \cdot|M|\cdot|Y^i_j|-C_\Gamma\cdot (|H^i_j|+1)^{k_0}\cdot |Y^i_j| }{|M|}=\left(\mu-\frac{C_\Gamma\cdot (|H^i_j|+1)^{k_0}}{|M|}\right)|Y^i_j|.$$
Let $\ell_0:=\max\{|\bar{y}_i|:1\leq i\leq n\}$. Define 
\begin{equation}\label{hm}
h_{M}:=\lceil\frac{\ell_0\cdot\log |M|}{-\log(1-\mu/2)}\rceil+1.
\end{equation}
Then there is some $N_{\mu/2}$ such that whenever $|M|\geq N_{\mu/2}$, we have
\begin{equation}\label{star}
\quad\frac{C_\Gamma\cdot (n\cdot h_{M}+\ell_0)^{k_0}}{|M|}\leq \frac{\mu}{2}.
\end{equation}
In particular, we have \begin{equation}\label{eq4}
\frac{C_\Gamma\cdot (n\cdot h_{M}+1)^{k_0}}{|M|}\leq \frac{\mu}{2}.
\end{equation}
Therefore, when $|H^i_j|\leq n\cdot h_{M}$, we have $|\varphi_i(h^i_{j+1};Y^i_{j})|\geq \frac{\mu}{2}|Y^i_j|$, and hence, $$|Y^i_{j+1}|=|Y^i_j|-|\varphi_i(h^i_{j+1};Y^i_{j})|\leq\left(1-\frac{\mu}{2}\right)|Y^i_j|.$$ Consequently, $$|Y^i_{j+1}|\leq\left(1-\frac{\mu}{2}\right)|Y^i_j| \leq \left(1-\frac{\mu}{2}\right)^2|Y^i_{j-1}|\leq\cdots\leq \left(1-\frac{\mu}{2}\right)^{j+1}|Y^i_0|\leq \left(1-\frac{\mu}{2}\right)^{j+1}\cdot |M|^{\ell_0}.$$

There is some $N_{\Delta,\Gamma}>\max\{N_{\mu/2},N\}$ such that whenever $|M|>N_{\Delta,\Gamma}$, we have $(1-\frac{\mu}{2})\cdot |M|>n\cdot h_{M}$. Fix some $M\in\mathcal{C}$ with $|M|>N_{\Delta,\Gamma}$ and let $$Y^1_0,\ldots,Y^1_{t_1};\cdots,;Y^i_0,\ldots,Y^i_{t_i}$$ be a maximal sequence. We claim that for each $i'\leq i$, if $|H^{i'}_{t_{i'}}|\leq n\cdot h_M$, then $t_{i'}\leq h_M$. Otherwise, $Y^{i'}_{h_M}$ is in the sequence. By the argument above, $|Y^{i'}_{h_{M}}|\leq (1-\frac{\mu}{2})^{h_{M}}\cdot |M|^{\ell_0}.$ By calculation, we have $$k>\frac{\ell_0\cdot \log |M|}{-\log(1-\mu/2)}\quad \Longrightarrow\quad \left(1-\frac{\mu}{2}\right)^{k}\cdot|M|^{\ell_0}<1.$$  Hence, $Y^{i_0}_{h_{M}}=\emptyset$. We conclude $t_{i_0}\leq h_{M}$. Therefore, $t_1\leq h_M$ and by induction, for each $1\leq i'\leq n$, we have $|H^{i'}_{t_{i'}}|=\sum_{1\leq j\leq i'} t_j\leq i'\cdot h_M$. Now we can see that $|H^i_{t_i}|\leq n\cdot h_M$.

Consider the set $X^i_{t_i+1}$. By inequality (\ref{Xij}), $$|X^i_{t_i+1}|\geq |M|-C_\Gamma\cdot (|H^i_{t_i}|+1)^{k_0}-|H^i_{t_i}|\geq |M|-C_\Gamma\cdot (n\cdot h_M+1)^{k_0}-n\cdot h_M.$$ By inequality (\ref{eq4}) and $(1-\frac{\mu}{2})\cdot |M|>n\cdot h_{M}$, we get $$|X^i_{t_i+1}|\geq |M|-\frac{\mu}{2}|M|-n\cdot h_{M}>0.$$ Hence $X^{i}_{t_i+1}\neq\emptyset$. As $Y^i_{t_i}$ is the end term of a maximal sequence, it can only be the case that $Y^i_{t_i}=\emptyset$ and $i=n$.

Therefore, if $|M|>N_{\Delta,\Gamma}$ and $M\in\mathcal{C}$, then $H_{\Delta,\Gamma}(M)$ exists and $$|H_{\Delta,\Gamma}(M)|\leq n\cdot h_{M}\leq C_{\Delta,\Gamma}\cdot \log|M|,$$ where $C_{\Delta,\Gamma}:=n\cdot \left(\lceil\frac{\ell_0}{-\log(1-\mu/2)}\rceil+1\right)$.
\end{proof}
Take any $M\in\mathcal{C}$ with $|M|\geq N_{\Delta,\Gamma}$, let $H_{\Delta,\Gamma}(M)$ as defined in Claim \ref{claim-cover} and for $h^i_j,h^t_m\in H_{\Delta,\Gamma}$, define $h^i_j\leq h^t_m$ if $i< t$ or $i=t$ and $j\leq m$. By construction we have $(H_{\Delta,\Gamma}(M),\leq)$ covers $\psi_{\varphi}(\bar{y})$ and avoids $\phi(x,\bar{y})$ in $M$ for any $\varphi\in \Delta$ and $\phi(x,\bar{y})\in\Gamma$.
\end{proof}

\begin{theorem}
Let $\mathcal{C}$ be a one-dimensional asymptotic class in a countable language $\mathcal{L}$. Let $\mathcal{M}:=\prod_{i\in I}M_i/\mathcal{U}$ be an infinite ultraproduct of members among $\mathcal{C}$. Then exact pseudofinite $H$-expansions of $\mathcal{M}$ exist.

\end{theorem}
\begin{proof}
Let $\{\varphi_i(x;\bar{y}_i), i\in\mathbb{N}\}$ be a list of all formulas in $\mathcal{L}$ such that $x$ is in one variable and $\bar{y}_i\neq\emptyset$ is a tuple of variables. For $n\in\mathbb{N}$, let $\Delta_n:=\{\varphi_i(x;\bar{y}_i):i\leq n\}$.

Let $\{\xi_i(x;\bar{z}_i):i\in\mathbb{N}\}$ be a list of all formulas such that $\xi_i(x;\bar{z}_i)$ is algebraic ($\bar{z}_i$ can be empty). Let $\Gamma_n:=\{\xi_i(x;\bar{z}_i):i\leq n\}$.
 
By Lemma \ref{lem-cover}, there are $N_{\Delta_n,\Gamma_n}\in\mathbb{N}$ such that for any $M\in\mathcal{C}$ with $|M|\geq N_{\Delta_n,\Gamma_n}$ there exists $(H_{\Delta_n,\Gamma_n}(M),\leq)$ with $H_{\Delta_n,\Gamma_n}(M)\subseteq M$ such that $H_{\Delta_n,\Gamma_n}(M)$ covers $\psi_{\varphi}(\bar{y})$ and avoids $\xi(x;\bar{z})$ in $M$ for all $\varphi\in\Delta_n$ and $\xi(x,\bar{z})\in\Gamma_n$.

For any $i\in I$, let $i_n:=\max\{n:|M_i|\geq N_{\Delta_n,\Gamma_n}\}$ (set $\max\emptyset=-\infty$). Define $H_i:=H_{\Delta_{i_n},\Gamma_{i_n}}(M_i)$ if $i_n\neq\infty$; otherwise let $H_i:=\emptyset$.

\begin{claim}\label{claim-H-expansion}
$(\mathcal{M},H(\mathcal{M})):=\prod_{i\in I}(M_i,H_i)/\mathcal{U}$ is an exact pseudofinite $H$-expansion of $\mathcal{M}$.
\end{claim}
\begin{proof}
We only need to show that $(\mathcal{M},H(\mathcal{M}))$ is an $H$-expansion of $\mathcal{M}$. We verify the conditions one by one.
\begin{enumerate}
\item
$\mathcal{M}\models \mbox{Th}_{\mathcal{L}}(\mathcal{M})$: 
clear.
\item
$H(\mathcal{M})$ is an $\mbox{acl}_{\mathcal{L}}$-independent subset: Suppose, towards a contradiction, that there are $\{a_0,a_1,\ldots,a_k\}$ which are not $\mbox{acl}_\mathcal{L}$-independent. We may assume that any proper subset of $\{a_0,a_1,\ldots,a_k\}$ is an $\mbox{acl}_{\mathcal{L}}$-independent set. Suppose for $0\leq t\leq k$, each $a_t:=(a_t^i)_{i\in I}/\mathcal{U}$. Let $O:=(i_0i_1\cdots i_k)$ be an ordering of $0,1,\ldots,k$. Define $$I_{O}:=\{j\in I:(a^j_{i_0},a^j_{i_1},\ldots,a^j_{i_k})\text{ is increasing in }(H_j,\leq)\}.$$ Let $A$ be the collections of all the orderings of $0,1,\ldots,k$. Since $A$ is finite and $I=\bigcup_{O\in A}I_O$, we have exactly one $I_O\in \mathcal{U}$. We may assume that $O=(0\cdots k)$. Suppose $a_i\in\mbox{acl}_{\mathcal{L}}(\{a_0,\ldots,a_k\}\setminus \{a_i\})$. By assumption, $a_i\not\in\mbox{acl}_{\mathcal{L}}(\{a_0,\ldots,a_k\}\setminus \{a_i,a_k\})$. Since $\mbox{acl}_{\mathcal{L}}$ satisfies the exchange property, we have $a_k\in \mbox{acl}_{\mathcal{L}}(a_0,\ldots,a_{k-1})$. Let $\varphi(x;z_{0},\ldots,z_{k-1})$ witness algebraicity (i.e., $\varphi(x;z_{0},\ldots,z_{k-1})$ is algebraic and $\mathcal{M}\models \varphi(a_k;a_{0},\ldots,a_{k-1})$). By the list of all algebraic formulas, $\varphi(x;z_0,\ldots,z_{k-1})=\xi_j(x;z_0,\ldots,z_{k-1}):=\xi_j(x;\bar{z}_j)$ for some $j$.

Let $J:=\{i\in I:i_n\geq j\}=\{i\in I:|M_i|\geq N_{\Delta_j,\Gamma_j}\}$. Since $\mathcal{M}$ is infinite, $J\in\mathcal{U}$. For any $i\in J$, we have $\xi_j(x;\bar{z}_j)\in\Gamma_{i_n}$, hence $H_i$ avoids $\xi_j(x;\bar{z}_j)$. As $a^i_k>\max\{a^i_0,\ldots,a^i_{k-1}\}$ in $H_i$, by construction, the set $H_i$ avoids $\xi_j(x;\bar{z}_j)$, we get $$M_i\models \neg\xi_j(a^i_k;a^i_0,\ldots,a^i_{k-1})$$ for any $i\in J$. We conclude $\mathcal{M}\models\neg\xi_j(a_k;a_0,\ldots,a_{k-1})$, contradiction. 

\item
Density/coheir property: As $(\mathcal{M},H(\mathcal{M}))$ is pseudofinite, it is $\aleph_1$-saturated. Therefore, we only need to show that for any $a_0,\ldots,a_k\in\mathcal{M}$, if $\varphi(x;a_0,\ldots,a_k)$ is non-algebraic, then there is $h\in H(\mathcal{M})$ such that $\mathcal{M}\models \varphi(h;a_0,\ldots,a_k)$. We may assume that $\varphi(x;y_0,\ldots,y_k)=\varphi_j(x;\bar{y}_j)$. 

Let $J:=\{i\in I:i_n\geq j\}=\{i\in I:|M_i|\geq N_{\Delta_j,\Gamma_j}\}$. Then $J\in\mathcal{U}$. Note that $\varphi_j(x;\bar{y}_j)\in\Delta_{i_n}$ for any $i\in J$. Therefore $H_i$ covers $\psi_{\varphi_j}(\bar{y}_j)$ in $M_i$ for any $i\in J$. 

Suppose $a_t:=(a_t^i)_{i\in I}/\mathcal{U}$ for $0\leq t\leq k$. Let $$J':=\{i\in J:M_i\models \psi_{\varphi_j}(a^i_0,\ldots,a^i_k)\}.$$ As $\varphi_j(x;a_0,\ldots,a_k)$ is non-algebraic, $J'\in\mathcal{U}$.

For any $i\in J'$, since $H_i$ covers $\psi_{\varphi_j}(\bar{y}_j)$ in $M_i$ and $M_i\models \psi_{\varphi_j}(a^i_0,\ldots,a^i_k)$, there is some $h_i\in H_i$ such that $M_i\models \varphi_j(h_i;a^i_0,\ldots,a^i_k)$. For $i\not\in J'$, choose $h_i\in M_i$ randomly. Let $h:=(h_i)_{i\in I}/\mathcal{U}$. Then $h\in H(\mathcal{M})$ and $\mathcal{M}\models \varphi_j(h;a_0,\ldots,a_k)$, i.e., $\mathcal{M}\models \varphi(h;a_0,\ldots,a_k)$.

\item
Extension Property: Suppose $A\subseteq \mathcal{F}$ is finite dimensional. Let $A'=\{a_0,\ldots,a_k\}$ be a base of $A$. Suppose $a_t:=(a_t^i)_{i\in I}/\mathcal{U}$ for each $t\leq k$. Let $A'_i=\{a^i_0,\ldots,a^i_k\}\subseteq M_i$. Let $$\mbox{clos}_i(H_i\cup A'_i):=\bigcup_{j\leq i_n,~\bar{a}\in (H_i\cup A'_i)^{|\bar{z}_j|}}\xi_j(M_i;\bar{a}),$$ and define $\mbox{clos}(H(\mathcal{M}')\cup A'):=\prod_{i\in I}\mbox{clos}_i(H_i\cup A'_i)/\mathcal{U}$. By essentially the same argument as $\mbox{acl}_{\mathcal{L}}$-independence of $H(\mathcal{M})$, we have $$\mbox{acl}_{\mathcal{L}}(H(\mathcal{M})\cup A)\subseteq \mbox{clos}(H(\mathcal{M})\cup A').$$

By the fact that $(\mathcal{M},\mbox{clos}(H(\mathcal{M})\cup A'))$ is pseudofinite, hence $\aleph_1$-saturated, we only need to show that for any $b_0,\ldots,b_t\in A$, if $\varphi(x;b_0,\ldots,b_t)$ is non-algebraic, then there is $a\in \mathcal{M}\setminus \mbox{clos}(H(\mathcal{M})\cup A')$ such that $\mathcal{M}\models \varphi(a;b_0,\ldots,b_t)$. We may assume that $\varphi(x;y_0,\ldots,y_t)=\varphi_j(x;\bar{y}_j)$. 
Assume $b_k=(b^i_k)_{i\in I}/\mathcal{U}$ for $k\leq t$. There is some $J\in\mathcal{U}$ and $\mu>0$ such that for all $i\in J$, we have $|\varphi(M_i;b_0^i,\ldots,b_t^i)|\geq\mu\cdot |M_i|.$

Consider the size of $\mbox{clos}_i(H_i\cup A')$. We have $$|\mbox{clos}_i(H_i\cup A')|\leq C_{\Gamma_{i_n}}\cdot(|H_i\cup A'|)^{k_0},$$ where as above $\Gamma_{i_n}:=\{\xi_j(x;\bar{z}_j):j\leq i_n\}$, $k_0:=\max\{|\bar{z}_j|:j\leq i_n\}$ and $C_{\Gamma_{i_n}}:=(i_n+1)\cdot C$ with $C$ is the largest number of solutions of $\xi_j$ over parameters for $j\leq i_n$.

Let $\Delta_{i_n}:=\{\varphi_j(x;\bar{y}_j):j\leq i_n\}$ and $\ell_0:=\max\{|\bar{y}_j|:j\leq i_n\}$.
Note that there is some $J'\in\mathcal{U}$ such that for all $i\in J'$ we have $k\leq \ell_0$. Hence $$|H_i\cup A'|\leq |H_i|+k\leq |\Delta_{i_n}|\cdot h_{M_i}+\ell_0,$$ where $h_{M_i}$ is defined as the equation (\ref{hm}). By the inequality (\ref{star}), we have $$C_{\Gamma_{i_n}}\cdot(|\Delta_{i_n}|\cdot h_{M_i}+\ell_0)^{k_0}\leq \frac{\mu}{2}\cdot|M_i|.$$ Therefore, $$|\mbox{clos}_i(H_i\cup A')|\leq C_{\Gamma_{i_n}}\cdot(|H_i\cup A'|)^{k_0}\leq\frac{\mu}{2}\cdot|M_i|,$$ for all $i\in J\cap J'$.

As $|\varphi(M_i;b_0^i,\ldots,b_t^i)|\geq\mu\cdot |M_i|$, there must be some $$a_i\in \varphi(M_i;b_0^i,\ldots,b_t^i)\setminus \mbox{clos}_i(H_i\cup A')$$ for all $i\in J\cap J'$. Choose $a_i$ at random for $i\not\in J\cap J'$. Set $a:=(a_i)_{i\in I}/\mathcal{U}$, then $a\not\in \mbox{clos}(H\cup A')$ and $\mathcal{M}\models\varphi(a;b_0,\ldots,b_t)$.\qedhere
\end{enumerate}
\end{proof}

\end{proof}

\begin{corollary}
Let $\mathcal{C}$ be a one-dimensional asymptotic class in a language $\mathcal{L}$ and $\mathcal{M}$ be an infinite ultraproduct of members of $\mathcal{C}$. Suppose $\mbox{acl}_{\mathcal{L}}$ of $Th_{\mathcal{L}}(\mathcal{M})$ is non-trivial. Then the exact pseudofinite $H$-expansion $(\mathcal{M},H(\mathcal{M}))$ is a pseudofinite structure whose theory is supersimple of $SU$-rank $\omega$.
\end{corollary}

\noindent\emph{Remark:} Let $\mathcal{M}:=\prod_{i\in I}M_i/\mathcal{U}$ be an infinite ultraproduct of a one-dimensional asymptotic class. We can also make the $H$-expansion $(\mathcal{M},H(\mathcal{M})):=\prod_{i\in I} (M_i,H_i)/\mathcal{U}$ satisfying $$\lim_{i\in I}\frac{\log|H_i|}{\log |M_i|}=0~~\text{ that is }~~\pmb{\delta}_{\mathcal{M}}(H(\mathcal{M}))=0,$$ that is  the pseudofinite coarse dimension of $H(\mathcal{M})$ with respect to $\mathcal{M}$ is zero.

This is because by Lemma \ref{lem-cover} we know that $|H_i|=C_{\Delta_{i_n},\Gamma_{i_n}}\cdot \log|M_i|$ where $C_{\Delta_{i_n},\Gamma_{i_n}}$ depends only on $\Delta_{i_n}$ and $\Gamma_{i_n}$. If we redefine $$i_n:=\max\{n:|M_i|>N_{\Delta_n,\Gamma_n}\text{ and }|M_i|>(C_{\Delta_{n},\Gamma_{n}})^n\},$$ we see that additionally $\pmb{\delta}_{\mathcal{M}}(H(\mathcal{M}))=0$. 

Note that for generic element $m\in M$, we have $SU_H(m)=\omega$ while $SU_H(h)<\omega$ for any element $h\in H(M)$. In a following project, together with other collaborators, we found this fact generalises to all definable sets. That is, the coarse dimension of a definable set equals to the coefficient of the $\omega$-part of the SU-rank of generic elements. We also wonder if $(M_i)_{i\in I}$ is a one-dimensional asymptotic class, then the class $(M_i,H_i)_{i\in I}$ we build in Claim \ref{claim-H-expansion} forms a multidimensional asymptotic class. We expect this should involve a more detailed treatment of definable sets in $H$-structures.

\section{Groups in H-structures}\label{groups}
This section deals with definable groups in $H$-structures when the base theory is supersimple of $SU$-rank one. We ask whether there are any new definable groups in $H$-structures.
As we said before, in \cite{Berenstein2016} the authors have partially solved the question by showing that in stable theories the connected component of an $\mathcal{L}_H$-definable group in an $H$-structure is isomorphic to some $\mathcal{L}$-definable group. We record their results here.

\begin{fact}\label{BV-gp1}(\cite[Proposition 6.5]{Berenstein2016})

Let $D$ be a group in a language $\mathcal{L}$ with $RM(D)=1$ and assume that $(D,H)$ is an $\aleph_0$-saturated $H$-structure. Let $A\subseteq D$ be finite and let $G\leq D^n$ be an $\mathcal{L}_H$-definable subgroup defined over $A$. Then $G$ is $\mathcal{L}$-definable over $A$.
\end{fact}

\begin{fact}\label{BV-gp2}(\cite[Proposition 6.6]{Berenstein2016})

Let $M$ be a stable structure of $U$-rank one in a language $\mathcal{L}$ and let $H$ be a subset of $M$ such that $(M,H)$ is an $\aleph_1$-saturated $H$-structure. Let $A\subseteq M$ be countable and let $G\subseteq M^n$ be an $\mathcal{L}_H$-definable group over $A$. Let $G^0$ be the connected component of $G$. Then $G^0$ is definably isomorphic to an $\mathcal{L}$-definable group over $A$.
\end{fact}

In this section, we will show that in supersimple theories, all $\mathcal{L}_H$-definable groups in $H$-structures are definably isomorphic to $\mathcal{L}$-definable groups.\footnote{Indeed, we need to assume that the base theory has elimination of imaginaries. Fact \ref{BV-gp1} and \ref{BV-gp2} also have this assumption.}

We first introduce some basic notions and facts about $H$-structures developed in \cite{Berenstein2016}, as well as some results about groups in simple theories that we will use later. 

Let $(M,H(M))$ be an $H$-structure. To simplify the notation, we write with subscript/superscript $H$ for notions in $T_H:=Th_{\mathcal{L}_H}(M,H(M))$ and no subscript/superscript for $T=Th_{\mathcal{L}}(M)$. We also write $\mathcal{L}$-independent to denote forking independence in $T$ ($\mathcal{L}_H$-independent for $T_H$ respectively), and $\mathcal{L}$-generic for generic group element in $T$ ($\mathcal{L}_H$-generic for $T_H$ respectively).

\begin{definition}\label{def-12}
Let $A$ be a subset of an $H$-structure $(M,H(M))$. We say that $A$ is $H$-independent if $A\ind_{A\cap H(M)}H(M)$.
\end{definition}
\noindent\emph{Remark:} Note that this is not the same as being $\mathcal{L}_H$-independent in the sense of forking in $T_H$.

\begin{definition}
Let $a$ be a tuple in an $H$-structure $(M,H(M))$ and let $C=\mbox{acl}(C)$ be $H$-independent. Define the \textsl{$H$-basis of $a$ over $C$}, denoted by $HB(a/C)$, as the smallest tuple $h$ in $H(M)$ such that $a\ind_{C,h}H(M)$. 
\end{definition}
By \cite[Proposition 3.9]{Berenstein2016}, $H$-bases exist and are unique up to permutation. Here is a useful observation:
\begin{lemma} Let $(M,H(M))$ be an $H$-structure and $a$ be a tuple. Suppose a subset $C=\mbox{acl}(C)$ is $H$-independent and $HB(a/C)=\emptyset$. Then  $HB(a,C)=HB(C)$.
\end{lemma}
\begin{proof}
Suppose not, then $a,C\nind_{HB(C)}H(M)$. There is a finite tuple $c\subseteq C$ such that $a,c\nind_{HB(C)}H(M)$. Denote the dimension of the underlying geometric theory as $\mbox{dim}_{\mbox{acl}}$.  Let $c'\subseteq C$ be a finite tuple such that $\mbox{dim}_{\mbox{acl}}(a/C)=\mbox{dim}_{\mbox{acl}}(a/c')$. Let $c''\subseteq C$ be a tuple containing both $c$ and $c'$. Then $\mbox{dim}_{\mbox{acl}}(a,c''/HB(C))>\mbox{dim}_{\mbox{acl}}(a,c''/H(M))$. By the choice of $c''$, we have $$\mbox{dim}_{\mbox{acl}}(a/c'')\geq\mbox{dim}_{\mbox{acl}}(a/c'',HB(C))\geq\mbox{dim}_{\mbox{acl}}(a/C)=\mbox{dim}_{\mbox{acl}}(a/c'').$$ By assumption, $\mbox{dim}_{\mbox{acl}}(a/C,H(M))=\mbox{dim}_{\mbox{acl}}(a/C)$. Therefore, $$\mbox{dim}_{\mbox{acl}}(a/c'')\geq \mbox{dim}_{\mbox{acl}}(a/c'',H(M))\geq \mbox{dim}_{\mbox{acl}}(a/C,H(M))=\mbox{dim}_{\mbox{acl}}(a/C)=\mbox{dim}_{\mbox{acl}}(a/c'').$$ We conclude that $\mbox{dim}_{\mbox{acl}}(a/c'',H(M))=\mbox{dim}_{\mbox{acl}}(a/c'')=\mbox{dim}_{\mbox{acl}}(a/c'',HB(C))$.
Since $C$ is $H$-independent, we also have $\mbox{dim}_{\mbox{acl}}(c''/H(M))=\mbox{dim}_{\mbox{acl}}(c''/HB(C))$. By additivity of $\mbox{dim}_{\mbox{acl}}$, we have $$\begin{aligned}
\mbox{dim}_{\mbox{acl}}(a,c''/H(M))=\mbox{dim}_{\mbox{acl}}(a/c'',H(M))+\mbox{dim}_{\mbox{acl}}(c''/H(M))\\
=\mbox{dim}_{\mbox{acl}}(a/c'',HB(C))+\mbox{dim}_{\mbox{acl}}(c''/HB(C))=\mbox{dim}_{\mbox{acl}}(a,c''/HB(C)),
\end{aligned}$$ a contradiction.
\end{proof}

\begin{definition}
Let $M$ be a structure. A set $X$ is \textsl{hyper-definable} over $A\subseteq M$ if there is a type-definable set $Y\subseteq M^n$ for some $n\in\mathbb{N}$ and a type-definable equivalence relation $E$ on $Y$ both defined over $A$ such that $X=Y/E$.
\end{definition}

\begin{fact}\label{prop-H} \cite[Lemma 2.8, Corollary 3.14, Proposition 6.2]{Berenstein2016}

Let $(M,H(M))$ be an $H$-structure. 
\begin{enumerate}
\item
Let $a,b$ be $H$-independent tuples such that $\mbox{tp}(a,HB(a))=\mbox{tp}(b,HB(b))$. Then $\mbox{tp}_{H}(a)=\mbox{tp}_{H}(b)$.
\item
Let $A$ be a subset of $M$, then $\mbox{acl}_{H}(A)=\mbox{acl}(A,HB(A))$.
\item
Suppose $Th(M)$ is superrosy of thorn-rank one and $(M,H(M))$ is $\aleph_0$-saturated. Let $D$ be an $\mathcal{L}_H$-definable group over some finite $H$-independent set $A$. Let $b$ be a generic element of the group. Then $HB(b/A)=\emptyset$.
\end{enumerate}
\end{fact}


\begin{fact}\cite[Proposition 5.6]{Berenstein2016} Let $(M,H(M))\models T_H$ be a $\kappa$-saturated $H$-structure and $C\subseteq D\subseteq M$ be $\mbox{acl}_H$-closed and $\max\{|C|,|D|\}<\kappa$. Suppose $T$ is supersimple of SU-rank one and $a\in M$. Then $a\ind^{\!\! H}_C D$ if and only if none of the following holds:
\begin{itemize}
\item
$a\in D\setminus C$;
\item
$a\in\mbox{acl}(H(M),D)\setminus \mbox{acl}(H(M),C)$;
\item
$HB(a/C)\neq HB(a/D)$.
\end{itemize}
\end{fact}

\begin{fact}\label{prop-XX}\cite[Lemma 4.4.8]{Wagner-Supersimple}
Let $G$ be a type-definable/hyper-definable group in a simple theory. Let $X$ be a non-empty type-definable/hyper-definable subset of $G$. Suppose for independent $g,g'\in X$ we have $g^{-1}\cdot g'\in X$, and put $Y=X\cdot X$. Then $Y$ is a type-definable/hyper-definable subgroup of $G$, and $X$ is generic in $Y$. In fact, $X$ contains all generic types for $Y$.
\end{fact}

\begin{fact}\label{prop-gpchunk}\cite[Theorem 4.7.1]{Wagner-Supersimple}
We fix an ambient simple theory. Let $\pi$ be a partial type and $\star$ be a partial type-definable function defined on pairs of independent realizations of $\pi$, both over $\emptyset$ such that 
\begin{enumerate}
\item
Generic independence: for independent realizations $a,b$ of $\pi$ the product $a\star b$ realizes $\pi$ and is independent from $a$ and from $b$;
\item
Generic associativity: for three independent realizations $a,b,c$ of $\pi$, we have $(a\star b)\star c=a\star(b\star c)$; 
\item
Generic surjectivity: for any independent $a,b$ realizing $\pi$, there are $c$ and $c'$ independent from $a$ and from $b$, with $a\star c=b$ and $c'\star a=b$.
\end{enumerate}
Then there are a hyper-definable group $G$ and a hyper-definable bijection from $\pi$ to the generic types of $G$, such that generically $\star$ is mapped to the group multiplication. $G$ is unique up to definable isomorphism.
\end{fact}

We proceed by some lemmas, most of which are about the properties of generic elements of definable groups in $H$-structures. 

In the following we will assume $\kappa$ is an cardinal with $\kappa\geq |\mathcal{L}|$.


\begin{lemma}\label{lem-dcl}
Let $(M,H(M))$ be a $\kappa$-saturated $H$-structure such that $Th(M)$ is supersimple of $SU$-rank one. Let $G$ be an $\mathcal{L}_H$-(type-)definable group over some set $A$ with $|A|<\kappa$ and $\mbox{acl}_{H}(A)=A$. Let $a,b$ be $\mathcal{L}_H$-independent and $\mathcal{L}_H$-generic elements in $G$. Then $a\cdot b\in\mbox{dcl}(a,b,A)$ and $a^{-1}\in\mbox{dcl}(a,A)$.  
\end{lemma}
\begin{proof}
By Fact \ref{prop-H} (3), $HB(a/A)=HB(b/A)=\emptyset$. That is $a\ind_{A}H(M)$ and $b\ind_{A}H(M)$.

By assumption, $a\ind^{\!\!H}_A b$. Hence, $a\ind_{A,H(M)}b$. Thus, $a\ind_{A,H(M)}bH(M)$. Together with $a\ind_{A}A,H(M)$, we get $a\ind_{A}b,H(M)$. Hence, $a,b\ind_{A,b}H(M)$. Again, as $b\ind_{A}H(M)$, we have $a,b\ind_{A}H(M)$. Since $A\ind_{HB(A)}H(M)$, we conclude that $a,b,A\ind_{HB(A)}H(M)$. Therefore, $HB(a,b,A)\subseteq HB(A)\subseteq A$.

As $c:=a\cdot b\in\mbox{acl}_{H}(a,b,A)=\mbox{acl}(a,b,A,HB(a,b,A))=\mbox{acl}(a,b,A)$, we have $$a,b,c,A\ind_{HB(A)}H(M).$$ Take $c'\in M$ with $\mbox{tp}(c'/a,b,A)=\mbox{tp}(c/a,b,A)$. As $c'\in\mbox{acl}(a,b,A)$, we still have $a,b,c',A\ind_{HB(A)}H(M).$  Therefore, $a,b,c,A$ and $a,b,c',A$ are $H$-independent tuples of the same $\mathcal{L}$-type. By Fact \ref{prop-H} (1), $\mbox{tp}_H(a,b,c'/A)=\mbox{tp}_H(a,b,c/A)$. As $c$ is in the $\mathcal{L}_H$-definable closure of $a,b,A$, we get $c'=c$. Hence, $c\in\mbox{dcl}(a,b,A)$ as we have claimed.

The proof of $a^{-1}\in\mbox{dcl}(a,A)$ is similar.
\end{proof}

\begin{lemma}\label{lem-indreal}
Let $(M,H(M))$ be a $\kappa$-saturated model of $T_H$. Let $G\subseteq M^n$ be an $\mathcal{L}_H$-type-definable group over $A$ with $\mbox{acl}_{H}(A)=A$ and $|A|<\kappa$. Then there are a partial $\mathcal{L}_H$-type $\pi_G(x)$ and a partial $\mathcal{L}$-type $\pi_{\mathcal{L}}(x)$ over $A$ such that:
\begin{enumerate}
\item
$\pi_G(M^n)$ is the set of all $\mathcal{L}_H$-generics in $G$.
\item
For any complete $\mathcal{L}$-type $q(x)$ over $A$ with $q(x)\supseteq \pi_{\mathcal{L}}(x)$, there is a complete $\mathcal{L}_H$-type $p(x)$ over $A$ such that $p(x)\supseteq q(x)\cup\pi_G(x)$;
\item
Let $a,b,c$ be three realizations of $\pi_\mathcal{L}(x)$ over $A$. Then there are $a',b',c'\in G$ such that $a',b',c'$ realise $\pi_G(x)$, $HB(a',b',c'/A)=\emptyset$ and $\mbox{tp}(a,b,c/A)=\mbox{tp}(a',b',c'/A)$. In addition, if $a,b,c$ are $\mathcal{L}$-independent, then $a',b',c'$ are $\mathcal{L}_H$-independent.
\end{enumerate}
\end{lemma}

\begin{proof}
Suppose $G$ is defined by a partial type $\delta(x)$. Let $\pi_G(x)$ be the partial $\mathcal{L}_H$-type over $A$ which contains $\delta(x)$ and is closed under implication such that for all $a\in M^n$, $a\models\pi_G(x)$ if and only if $a$ is $\mathcal{L}_H$-generic in $G$. Let $\pi_\mathcal{L}(x)\subseteq \pi_G(x)$ be the restriction of $\pi_G(x)$ in the language $\mathcal{L}$. 

Claim: Item 2 holds. If not, then there exists $\mathcal{L}$-type $q(x)$ over $A$ extending $\pi_{\mathcal{L}}(x)$ such that $q(x)\cup \pi_G(x)$ is inconsistent. By compactness, there is some $\psi(x)\in q(x)$ such that $\pi_G(x)\vdash \neg\psi(x)$. As $\pi_G(x)$ is closed under implication, $\neg\psi(x)\in \pi_G(x)$, hence also $\neg\psi(x)\in\pi_\mathcal{L}(x)$, which contradicts that $q(x)\supseteq\pi_{\mathcal{L}}(x)$. 

Now we prove item 3. Write $a=(a_1,a_2)$, $b=(b_1,b_2)$ and $c=(c_1,c_2)$, where $SU(a_1/A)=|a_1|$, $a_2\in \mbox{acl}(a_1,A)$; $SU(b_1/A,a)=|b_1|$, $b_2\in\mbox{acl}(b_1,a,A)$ and $SU(c_1/A,a,b)=|c_1|$, $c_2\in\mbox{acl}(c_1,a,b,A).$ (We remark that $b_1,c_1$ can be empty.) As $SU(a_1,b_1,c_1/A)=|a_1|+|b_1|+|c_1|$ and $T$ has $SU$-rank 1,  we get $a_1,b_1,c_1$ are $\mathcal{L}$-independent.  By the axioms of of $T_H$ and $\kappa$-saturation, there are $a_1',b_1',c_1'$ in $M$ such that $\mbox{tp}(a_1,b_1,c_1/A)=\mbox{tp}(a_1',b_1',c_1'/A)$ and $$a_1',b_1',c_1'\ind_{A} H(M).$$ Let $a_2',b_2',c_2'$ be such that $$\mbox{tp}(a_1',a_2',b_1',b_2',c_1',c_2'/A)=\mbox{tp}(a_1,a_2,b_1,b_2,c_1,c_2/A).$$ Define $a':=(a_1',a_2')$, $b':=(b_1',b_2')$ and $c':=(c_1',c_2')$. 

Since $a_1',b_1',c_1'\ind_{A}H(M)$ and $a',b',c'\in \mbox{acl}(a_1',b_1',c_1',A)$, we get $a',b',c'\ind_A H(M)$. Therefore, $HB(a',b',c'/A)=\emptyset$. Hence, $HB(a'/A)=HB(b'/A)=HB(c'/A)=\emptyset$.

We only need to show that $a',b'$ and $c'$ satisfy $\pi_G(x)$. Let $q(x):=\mbox{tp}(a/A)\supseteq \pi_{\mathcal{L}}(x)$. By item 2, there is a complete $\mathcal{L}_H$-type $p(x)$ over $A$ extending $q(x)\cup\pi_G(x)$. Let $a''$ be a realization of $p(x)$. By Fact \ref{prop-H} (3), $HB(a''/A)=\emptyset$. Therefore, both $a',A$ and $a'',A$ are $H$-independent and $$\mbox{tp}(a',A,HB(a',A))=\mbox{tp}(a,A,HB(A))=\mbox{tp}(a'',A,HB(a'',A)).$$ By Fact \ref{prop-H} (1), $\mbox{tp}_H(a'/A)=\mbox{tp}_H(a''/A)$. Hence $\mbox{tp}_H(a'/A)\supseteq \pi_G(x)$. Similarly, $b'$ and $c'$ are realizations of $\pi_G(x)$.

In addition, if $a,b,c$ are $\mathcal{L}$-independent, then $b'=(b_1',b_2')$ and $c'=(c_1',c_2')$ are such that $SU(b_1'/A)=SU(b_1'/A,a')=|b_1'|$, $SU(c_1'/A)=SU(c_1'/A,a',b')=|c_1'|$ and $b_2'\in\mbox{acl}(b_1',A)$, $c_2'\in\mbox{acl}(c_1',A)$. As $a_1',b_1',c_1'\ind_{A} H(M)$ and $a_1'\ind_{A}b_1',c_1'$, we get $$a_1'\ind_{A} b_1',c_1',H(M).$$ Therefore,  $a'\ind_{A}b',c',H(M)$, whence $a'\ind_{AH(M)}b',c',H(M)$. Together with $HB(a'/A)=HB(a'/Ab'c')=\emptyset$ we get $a'\ind^{\!\!H}_A b',c'$. The other $\mathcal{L}_H$-independences among $a',b',c'$ are similar. Hence, $a',b',c'$ are $\mathcal{L}_H$-independent.
\end{proof}

\begin{lemma}\label{lem-hyperdef}
Let $\mathcal{L}_0\subseteq\mathcal{L}_1$ be two languages. Let $M$ be an $\mathcal{L}_1$-structure. Suppose $Y$ is $\mathcal{L}_0$-hyper-definable and $G$ is $\mathcal{L}_1$-type-definable in $M$ such that there is an $\mathcal{L}_1$-isomorphism from $Y$ to $G$, then $Y$ is $\mathcal{L}_0$-type-interpretable.
\end{lemma}
\begin{proof}
Suppose $G=\bigcap_{i\in I}G_i$ is $\mathcal{L}_1$-type-definable, $Y=X/R$ where $X=\bigcap_{i\in I}X_i$ and $R=\bigcap_{i\in I}R_i$ are $\mathcal{L}_0$-type-definable and $\Phi(x,y):=\bigcap_{i\in I}\Phi_i:X_i\to G_i$ is $\mathcal{L}_1$-type-definable which induces an isomorphism between $Y$ and $G$.

As $\Phi$ is the graph of a function from $X$ to $G$, we have:
$$\bigwedge_{i,j,k\in I}X_i(x)\land G_j(y)\land G_j(y')\land \Phi_k(x,y)\land \Phi_k(x,y')\models y=y'.$$
By compactness, there are some $i_0,\ldots,i_k$ such that $$f(x,y):=\bigcap_{j\leq k}\Phi_{i_j}(x,y)~\subseteq~\left(\bigcap_{j\leq k}X_{i_j}\times \bigcap_{j\leq k}G_{i_j}\right)$$ is an $\mathcal{L}_1$-definable graph of a partial function.

Let $R'\subseteq \left(\bigcap_{ j\leq k}X_{i_j}\right)\times\left(\bigcap_{ j\leq k}X_{i_j}\right)$ be the $\mathcal{L}_1$-definable equivalence relation given by $R'(x,x')$ if and only if there is some $g\in \bigcap_{ j\leq k}G_{i_j}$ such that both $f(x,g)$ and $f(x',g)$ hold. We claim that $$R'\upharpoonright (X\times X)=R.$$ Let $x,x'\in X$. Suppose $R(x,x')$ holds. As $\Phi$ is an isomorphism between $Y$ and $G$, there is some $g\in G$ with $\Phi(x,g)$ and $\Phi(x',g)$. Therefore, both $f(x,g)$ and $f(x',g)$ hold and so does $R'(x,x')$. On the other hand, if $R'(x,x')$ holds, then there is $g\in \bigcap_{ j\leq k}G_{i_j}$ with $f(x,g)$ and $f(x',g)$. Let $g',g''\in G$ such that $\Phi(x,g')$ and $\Phi(x',g'')$. Thus, we also have $f(x,g')$ and $f(x',g'')$. Since $f$ is a partial function, $g=g'=g''$. Therefore, $R(x,x')$ holds.





As $R$ is defined by $\bigcap_{i\in I}R_i$, by compactness, there is some $\{j_0,\ldots,j_t\}\supseteq \{i_0,\ldots,i_k\}$ such that on $\left(\bigcap_{ i\leq t}X_{j_i}\right)\times\left(\bigcap_{ i\leq t}X_{j_i}\right)$ we have $$R_{\mathcal{L}_0}(x,x'):=\bigcap_{ i\leq t}R_{j_i}(x,x')\subseteq R'(x,x').$$ Thus, $R_{\mathcal{L}_0}$ is $\mathcal{L}_0$-definable and it agrees with $R$ on $X$. We have $$\begin{aligned}
\bigwedge_{i\in I}(X_i(x_1)\land X_i(x_2)\land X_i(x_3))\models R_{\mathcal{L}_0}(x_1,x_1)\\
\land(R_{\mathcal{L}_0}(x_1,x_2)\to R_{\mathcal{L}_0}(x_2,x_1))\\\land (R_{\mathcal{L}_0}(x_1,x_2)\land R_{\mathcal{L}_0}(x_2,x_3)\to R_{\mathcal{L}_0}(x_1,x_3)).
\end{aligned}$$
By compactness, there are $\{k_0,\ldots,k_m\}\supseteq\{j_0,\ldots,j_t\}$ such that $R_{\mathcal{L}_0}$ is an equivalence relation on $\bigcap_{t\leq m}X_{k_t}$. Therefore, $R$ is $\mathcal{L}_0$-definable. 
\end{proof}

We first consider $\mathcal{L}_H$-(type-)definable subgroups of $\mathcal{L}$-(type-)definable groups. We generalize Fact \ref{BV-gp1} to supersimple theories.

\begin{theorem}
Let $T$ be non-trivial of $SU$-rank one and let $(M,H(M))\models T_H$ be $\kappa$-saturated. Suppose $D$ is an $\mathcal{L}$-(type-)definable group and $G$ is an $\mathcal{L}_H$-(type-)definable subgroup of $D$, both defined over some set $A=\mbox{acl}_H(A)$ with $|A|<\kappa$. Then $G$ is $\mathcal{L}$-(type-) definable ovear $A$.
\end{theorem}
\begin{proof}
Suppose $D\subseteq M^n$. Let $\pi_G(x)$ and $\pi_{\mathcal{L}}(x)$ be defined as in Lemma \ref{lem-indreal} with $|x|=n$. Suppose $D$ is defined by the partial $\mathcal{L}$-type $\chi(x)$. As $\pi_G(x)$ is closed under implication, $\pi_G(x)\supseteq\chi(x)$. Therefore, $\pi_{\mathcal{L}}(x)\supseteq\chi(x)$.

By Fact \ref{prop-XX}, $G=\pi_G(M^n)\cdot\pi_G(M^n)$. We will show that $\pi_{\mathcal{L}}(M^n)$ also satisfies the conditions of Fact \ref{prop-XX} in $T$.

Let $X:=\pi_{\mathcal{L}}(M^n)$. Since $\chi(x)\subseteq\pi_{\mathcal{L}}(x)$, we have $X\subseteq D$. Take two $\mathcal{L}$-independent realizations $a,b$ of $\pi_{\mathcal{L}}(x)$. By Lemma \ref{lem-indreal}, there are $a',b'$ both realising $\pi_G(x)$ such that $\mbox{tp}(a,b/A)=\mbox{tp}(a',b'/A)$ and $a'\ind_{A}^{\!\!H}b'$. Therefore, $(a')^{-1}\cdot b'$ is also generic in $G$, which implies $$\pi_\mathcal{L}(x)\subseteq \pi_G(x)\subseteq\mbox{tp}_H((a')^{-1}\cdot b'/A).$$ As $\mbox{tp}(a,b/A)=\mbox{tp}(a',b'/A)$ and group operations are $\mathcal{L}$-definable, we have $$\mbox{tp}(a^{-1}\cdot b/A)=\mbox{tp}((a')^{-1}\cdot b'/A).$$ Therefore, $\pi_{\mathcal{L}}(x)\subseteq \mbox{tp}(a^{-1}\cdot b/A)$, whence $a^{-1}\cdot b\in X$. By Fact \ref{prop-XX} we get an $\mathcal{L}$-type-definable group $D_G:=X\cdot X$ such that $X$ contains all $\mathcal{L}$-generics in $D_G$.

Clearly, $G\leq D_G$. Let $a$ be an $\mathcal{L}_H$-generic element in $D_G$. By Fact \ref{prop-H}(3), we have $HB(a/A)=\emptyset$.  Since $a$ is also $\mathcal{L}$-generic in $D_G$, we get $a\in X$. 
By Lemma \ref{lem-indreal} there is an $a'$ satisfying $\pi_G(x)$ such that $\mbox{tp}(a/A)=\mbox{tp}(a'/A)$. As $a'$ is $\mathcal{L}_H$-generic in $G$, $HB(a'/A)=\emptyset=HB(a/A)$. 
By Fact \ref{prop-H}(1), $\mbox{tp}_H(a'/A)=\mbox{tp}_H(a/A)$. Hence, $a$ realizes $\pi_G(x)$, i.e., $a$ is $\mathcal{L}_H$-generic in $G$. Therefore, every $\mathcal{L}_H$-generic element of $D_G$ is contained in $G$, whence $D_G\leq G$. We conclude that $G=D_G$. 

\end{proof}

Now we consider general $\mathcal{L}_H$-(type-)definable groups. The following is a generalization of Fact \ref{BV-gp2}.

\begin{theorem}
Let $T$ be supersimple of $SU$-rank one and $(M,H(M))\models T_H$ be $\kappa$-saturated. Let $G$ be an $\mathcal{L}_H$-(type-)definable group over a set $A=\mbox{acl}_H(A)$ of size less than $\kappa$. Then $G$ is $\mathcal{L}_H$-definably isomorphic to some $\mathcal{L}$-(type)-interpretable group. In particular, if $T$ eliminates imaginaries, then every $\mathcal{L}_H$-(type-)definable group is $\mathcal{L}_H$-definably isomorphic to some $\mathcal{L}$-(type-)definable group.
\end{theorem}

\begin{proof}
Suppose $G$ is type-definable. Let $\pi_G(x)$ and $\pi_{\mathcal{L}}(x)$ be defined as in Lemma \ref{lem-indreal}. In the following, we will extend $\mathcal{L}$-generically and $\mathcal{L}$-type-definably the group operation $\cdot$ of $G$ to $\star$ on $\pi_{\mathcal{L}}(x)$.

Let $\pi^2_G(x,y)\supseteq \pi_G(x)\cup \pi_G(y)$ be the partial $\mathcal{L}_H$-type over $A$ such that $a,b$ are $\mathcal{L}_H$-independent and $\mathcal{L}_H$-generic in $G$ over $A$ if and only if $(a,b)\models \pi^2_G(x,y)$ for any $a,b\in M^n$. For $(a,b)\models \pi^2_G(x,y)$, we have $a\cdot b\in \mbox{dcl}(a,b)$ by Lemma \ref{lem-dcl}. That is $a\cdot b=f_{a,b}(a,b)$  for some $\mathcal{L}$-definable function $f_{a,b}$ over $A$. Let $\mbox{dom}_{a,b}(x,y)$ be the $\mathcal{L}$-formula that defines the domain of the function $f_{a,b}$. Then define the $\mathcal{L}_H$-formula $$\varphi_{a,b}(x,y):=\mbox{dom}_{a,b}(x,y)\land x\cdot y=f_{a,b}(x,y).$$ Then we can see that $$\pi^2_G(x,y)\subseteq \bigcup_{(a,b)\models \pi_G^2(x,y)}\varphi_{a,b}(x,y).$$ By compactness, there are $(a_1,b_1),(a_2,b_2),\ldots,(a_k,b_k)$ such that $$\pi^2_G(x,y)\models \bigvee_{1\leq i\leq k}\varphi_{a_i,b_i}(x,y).$$

Let $(a,b)$, $(c,d)$ be two pairs of realizations of $\pi^2_G(x,y)$ such that $\mbox{tp}(a,b/A)=\mbox{tp}(c,d/A)$. Note that $(a,b)$ is an $\mathcal{L}_H$-generic element in $G\times G$. By Fact \ref{prop-H}(3), $HB(a,b/A)=\emptyset$. Similarly, $HB(c,d/A)=\emptyset$. Applying Fact \ref{prop-H}(1), we get $\mbox{tp}_{H}(a,b/A)=\mbox{tp}_{H}(c,d/A)$. Therefore, $(M,H(M))\models\varphi_{a_i,b_i}(a,b)\leftrightarrow \varphi_{a_i,b_i}(c,d)$ for all $1\leq i\leq k$. The above argument shows: 
$$\pi_G^2(x,y)\land\pi_G^2(x',y')\land\bigwedge_{\psi\in\mathcal{L}(A)}\psi(x,y)\leftrightarrow\psi(x',y')\models \bigwedge_{1\leq i\leq n}(\varphi_{a_i,b_i}(x,y)\leftrightarrow \varphi_{a_i,b_i}(x',y')).$$

By compactness, there is some finite set of $\mathcal{L}(A)$ formulas $\Delta$ such that the $\Delta$-type of any pair $(a,b)\models \pi^2_G(x,y)$  determines $(a,b)\models\varphi_{a_i,b_i}(x,y)$ or $(a,b)\models\neg\varphi_{a_i,b_i}(x,y)$ for any $1\leq i\leq k$. Hence, there are $\mathcal{L}$-formulas $\psi_1(x,y),\ldots,\psi_k(x,y)$ such that $$\pi^2_G(x,y)\models \bigvee_{1\leq i\leq k}\psi_i(x,y)$$ and for any $1\leq i\leq k$, we have $$ \pi^2_G(x,y)\models \psi_i(x,y)\to \left(\varphi_{a_i,b_i}(x,y)\land \bigwedge_{1\leq j<i}\neg\varphi_{a_j,b_j}(x,y)\right).$$

Let $\pi^2_{\mathcal{L}}(x,y)\supseteq \pi_{\mathcal{L}}(x)\cup\pi_\mathcal{L}(y)$ be the partial $\mathcal{L}$-type over $A$ such that $(a,b)\models \pi^2_{\mathcal{L}}(x,y)$ if and only if $a,b$ are $\mathcal{L}$-independent over $A$. By Lemma \ref{lem-indreal}, for $(a,b)\models \pi^2_{\mathcal{L}}(x,y)$, there are $a',b'$ realizing $\pi_{G}(x)$ such that $a'\ind^{\!\!H}_A b'$ and $\mbox{tp}(a,b/A)=\mbox{tp}(a',b'/A)$. Note that $(a',b')\models \pi^2_G(x,y)$. Hence, $$(a',b')\models \psi_i(x,y)\land \varphi_{a_i,b_i}(x,y)\land \bigwedge_{1\leq j<i}\neg\varphi_{a_j,b_j}(x,y)$$ for some $1\leq i\leq k$. As $\mbox{tp}(a,b/A)=\mbox{tp}(a',b'/A)$, we also have $$(a,b)\models \psi_i(x,y)\land \mbox{dom}_{a_i,b_i}(x,y).$$ Define $a\star b:=f_{a_i,b_i}(a,b)$. As $f_{a_i,b_i}(a',b')\models \pi_{\mathcal{L}}(x)$ and $\mbox{tp}(a,b/A)=\mbox{tp}(a',b'/A)$, we also have $f_{a_i,b_i}(a,b)\models \pi_{\mathcal{L}}(x)$. Note that $a\star b$ is defined by $f_{a_i,b_i}(x,y)$ if and only if $(a,b)\models \psi_i(x,y)$. Hence, $\star
$ is an $\mathcal{L}$-type-definable function from $\pi^2_{\mathcal{L}}(M^n,M^n)$ to $\pi_{\mathcal{L}}(M^n)$ and $\star$ agrees with $\cdot$ on $\pi^2_G(M^n,M^n)$.

We now verify all the conditions of the group chunk theorem (Fact \ref{prop-gpchunk}) in order to obtain an $\mathcal{L}$-hyper-definable group out of the generically given group operation.

\begin{lemma}
The $\mathcal{L}$-type-definable function $\star:\pi^2_{\mathcal{L}}(M^n,M^n)\to\pi_{\mathcal{L}}(M^n)$ satisfies all the conditions in Fact \ref{prop-gpchunk}.
\end{lemma}
\begin{proof}
Generic independence: Let $a,b$ be $\mathcal{L}$-independent realizations of $\pi_{\mathcal{L}}(x)$ and $c:=a\star b$. Then there are $\mathcal{L}_H$-independent and $\mathcal{L}_H$-generic elements $a',b'$ over $A$ such that $\mbox{tp}(a',b'/A)=\mbox{tp}(a,b/A)$. Let $c':=a'\cdot b'$. Since $\star$ is $\mathcal{L}$-definable and agrees with $\cdot$ on $\pi^2_G(M^n,M^n)$, we get $c'=a'\star b'$. Therefore, $\mbox{tp}(a',b',c'/A)=\mbox{tp}(a,b,c/A)$. As $c'\ind^{\!\!H}_A a'$, we have $c'\ind_A a'$. Hence, we also have $c\ind_A a$. Similarly, $c\ind_A b$.

Generic associativity: Let $a,b,c$ be $\mathcal{L}$-independent realizations of $\pi_{\mathcal{L}}(x)$. By Lemma $\ref{lem-indreal}$, there are $\mathcal{L}_H$-generic and $\mathcal{L}_H$-independent realizations $a',b',c'$ such that $$\mbox{tp}(a,b,c/A)=\mbox{tp}(a',b',c'/A).$$ Now we have $$\mbox{tp}((a\star b)\star c),a\star(b\star c))=\mbox{tp}((a'\star b')\star c',a'\star(b'\star c'))=\mbox{tp}((a'\cdot b')\cdot c',a'\cdot(b'\cdot c')).$$
Since $(a'\cdot b')\cdot c'=a'\cdot(b'\cdot c')$ we get $(a\star b)\star c=a\star(b\star c)$.


Generic surjectivity: for any $\mathcal{L}$-independent realizations $a,b$ of $\pi_{\mathcal{L}}(x)$, there are $\mathcal{L}_H$-independent realizations $a',b'$ of $\pi_G(x)$ such that $\mbox{tp}(a,b/A)=\mbox{tp}(a',b'/A)$. Let $c':=(a')^{-1}\cdot b'$. Then $c'$ is $\mathcal{L}_H$-independent from $a'$ and from $b'$. By Lemma \ref{lem-dcl}, $c'\in \mbox{dcl}((a')^{-1},b',A)=\mbox{dcl}(a',b',A)$. Let $c$ be the element with $\mbox{tp}(a,b,c/A)=\mbox{tp}(a',b',c'/A)$. Clearly, $c$ realizes $\pi_{\mathcal{L}}(x)$ and is $\mathcal{L}$-independent from $a$ and from $b$. Since $a'\cdot c'=a'\star c'=b'$ and $\mbox{tp}(a,b,c/A)=\mbox{tp}(a',b',c'/A)$, we have $a\star c=b$. Similarly, we can find $c''$ realizing $\pi_{\mathcal{L}}(x)$, $\mathcal{L}$-independent from $a$ and from $b$ such that $c''\star a=b$.
\end{proof}

By Fact \ref{prop-gpchunk}, there are an $\mathcal{L}$-hyper-definable group $D$ over $A$, and an $\mathcal{L}$-type-definable embedding $f:\pi_{\mathcal{L}}(M^n)\to D$ over $A$ such that $f(\pi_{\mathcal{L}}(M^n))$ contains all $\mathcal{L}$-generics of $D$. 

Consider $f(\pi_G(M^n))\subseteq D$. Take $g,g'$ $\mathcal{L}_H$-independent elements in $f(\pi_G(M^n))$. Suppose $g=f(a)$ and $g'=f(b)$. As $f$ is an $\mathcal{L}_H$-definable injection, we get $a\ind^{\!\! H}_A b$. Hence, $a^{-1}\star b\models\pi_G(x)$ and $a\ind^{\!\!H}_A a^{-1}\star b$. Since $f$ preserves $\star$ generically and $a,a^{-1},b\in G$, we have $$f(a)\cdot f(a^{-1}\star b)=f(a\star (a^{-1}\star b))=f(a\cdot(a^{-1}\cdot b))=f(b).$$ Hence, $f(a)^{-1}\cdot f(b)=f(a^{-1}\star b)\in f(\pi_G(M^n))$. By Fact \ref{prop-XX}, $$G_f:=f(\pi_G(M^n))\cdot f(\pi_G(M^n))$$ is an $\mathcal{L}_H$-hyper-definable group, and $f(\pi_G(x))$ contains all $\mathcal{L}_H$-generics in $G_f$. 

Let $X:=\{(g,f(g)):g\models \pi_G(x)\}\subseteq G\times G_f$. Let $(g_1,f(g_1))$ and $(g_2,f(g_2))$ be $\mathcal{L}_H$-independent tuples in $X$. Consider $$x_{g_1,g_2}:=(g_1,f(g_1))^{-1}\cdot (g_2,f(g_2))=(g_1^{-1},f(g_1^{-1}))\cdot(g_2,f(g_2))=(g_1^{-1}\star g_2,f(g_1^{-1}\star g_2)).$$ As $g_1\ind^{\!\! H}_A g_2$ in $\pi_G(x)$ we get $g_1^{-1}\star g_2=g_1^{-1}\cdot g_2\in \pi_G(x)$. Therefore, $x_{g_1,g_2}\in X$. By Fact \ref{prop-XX}, $C:=X\cdot X$ is a subgroup of $G\times G_f$. Consider the projection $\rho_1(C)\leq G$. It contains $\pi_G(M^n)$, hence contains all $\mathcal{L}_H$-generics of $G$. Thus $\rho_1(C)=G$. Similarly, $\rho_2(C)=G_f$. Let $I:=\{g:(g,\mbox{1})\in C\}$ and $I':=\{g:(\mbox{1},g)\in C\}$. If $g\in I$, then there are $g_1,g_2\in \pi_G(M^n)$ such that $g=g_1\star g_2$ and $f(g_1)\cdot f(g_2)=f(g_1\star g_2)=1$. As $f$ is an embedding, we get $g_1\star g_2=\mbox{1}$. Therefore, $I=\{\mbox{1}\}$. Similarly, $I'=\{\mbox{1}\}$. Hence, $C$ is the graph a group isomorphism between $G$ and $G_f$.

Let $a$ be an $\mathcal{L}_H$-generic in $D$. Then $HB(a/A)=\emptyset$. Since $a$ is also $\mathcal{L}$-generic in $D$, we get that $f^{-1}(a)$ satisfies $\pi_{\mathcal{L}}(x)$. As $f$ is an $\mathcal{L}_H$-definable embedding, we have 
$HB(f^{-1}(a)/A)=\emptyset$. 
Since $f^{-1}(a)\models \pi_{\mathcal{L}}(x)$, by Lemma \ref{lem-indreal} there is $a'$ realizing $\pi_G(x)$ such that $a'$ and $f^{-1}(a)$ have the same $\mathcal{L}$-type over $A$. Note that $HB(a'/A)=\emptyset$. By Fact \ref{prop-H} (1), $\mbox{tp}_H(a'/A)=\mbox{tp}_H(f^{-1}(a)/A)$. Hence, $f^{-1}(a)$ realizes $\pi_G(x)$, and $a=f(f^{-1}(a))$ is $\mathcal{L}_H$-generic in $G_f$. Therefore, the set of $\mathcal{L}_H$-generics of $D$ is contained in $G_f$, whence $D\leq G_f$. Together with $G_f\leq D$, we get $G_f=D$ and $G$ is $\mathcal{L}_H$-type-definably isomorphic to $D$.

Now Lemma \ref{lem-hyperdef} implies that $D$ is $\mathcal{L}$-type-interpretable.

Suppose $D=D_G/E$ where $E$ is an $\mathcal{L}$-definable equivalence relation and $D_G$ is $\mathcal{L}$-type-definable. If $G$ is definable, then $D_G$ is the image of an $\mathcal{L}_H$-definable function, hence $\mathcal{L}_H$-definable. By compactness $D_G$ is $\mathcal{L}$-definable. Therefore, $G$ is $\mathcal{L}_H$-definably isomorphic to an $\mathcal{L}$-interpretable group $D$.
\end{proof}

\noindent\emph{Remark:} Given an $\mathcal{L}_H$-definable group $G$, without the assumption that $G$ lives inside an $\mathcal{L}$-definable group, we cannot generally have that $G$ is $\mathcal{L}$-definable. Here is an example. 

\begin{example}
Let $D=(D,\cdot,{}^{-1})$ be a group without involutions of $SU$-rank one in the language $\mathcal{L}=\{\cdot, {}^{-1}\}$. Let $(D,H(D))$ be an $H$-structure. 

Define $\sigma:D\to D$ as $\sigma(x)=x$ if $x\not\in H(D)\cup (H(D))^{-1}$; and $\sigma(x)=x^{-1}$ if $x\in H(D)\cup (H(D))^{-1}$. Let $\star:G\times G\to G$ be defined as $a\star b:=\sigma^{-1}(\sigma(a)\cdot\sigma(b))$. Then the group $(D,\star,{}^{-1})$ is $\mathcal{L}_H$-isomorphic to $(D,\cdot,{}^{-1})$ via $\sigma$, but not $\mathcal{L}$-definable.
\end{example}

\renewcommand{\abstractname}{Acknowledgements}
\begin{abstract}
The author wants to thank her supervisor Frank Wagner for proposing the interesting question that initiated this project, for contributing at least half of the ideas of this paper. She also wants to thank the referee for the valuable comments and corrections. She is thankful to Gareth Boxall for suggesting looking at $H$-structures; to Jan Dobrowolski and Dugald Macpherson for pointing out that the first main result of this paper generalises from pseudofinite fields to one-dimensional asymptotic classes.
\end{abstract}

\label{Bibliography}

\bibliographystyle{plain} 
\bibliography{dim1.bib}

\end{document}